\documentclass[a4paper,12pt]{article}
\usepackage[text={180mm,267mm},centering]{geometry} 
\usepackage{amssymb,amsmath,amsthm,graphics}
\newtheorem{thm}{Theorem}
\newtheorem{cor}{Corollary}
\newtheorem{lem}{Lemma}
\newtheorem{dfn}{Definition}

\begin{document}
\begin{center}
  \textbf{\Large On Some Algorithmic and Structural Results on Flames} 
\end{center}

\bigskip

\begin{center}
  \begin{tabular}{l}
    \textbf{D\'avid Szeszl\'er}\\[\bigskipamount]
    Department of Computer Science and Information Theory\\
    Budapest University of Technology and Economics\\
    Budapest, Hungary\\[\medskipamount]
    email: \texttt{szeszler@cs.bme.hu}
  \end{tabular}
\end{center}

\bigskip

\begin{quote}
\textbf{Abstract.} A directed graph $F$ with a root node $r$ is
  called a \emph{flame} if for every vertex $v$ other than $r$ the
  local edge-connectivity value $\lambda(r,v)$ from $r$ to $v$ is
  equal to $\varrho_F(v)$, the in-degree of $v$. It is a classic,
  simple and beautiful result of Lov\'asz \cite{Lovasz73} that every
  digraph $D$ with a root node $r$ has a spanning subgraph $F$ that is
  a flame and the $\lambda(r,v)$ values are the same in $F$ as in $D$
  for every vertex $v$ other than $r$. However, the complexity of
  finding the minimum weight of such a subgraph is open
  \cite{Joo21}. In this paper we prove that this problem is solvable
  in strongly polynomial time for acyclic digraphs. Besides that, we
  prove a decomposition result of flames into a chain of smaller
  flames via edge-disjoint branchings and use this to prove a common
  generalization of Lov\'asz's above mentioned theorem and Edmonds'
  classic disjoint arborescences theorem.
\end{quote}

\medskip

\begin{quote}
\end{quote}

\section{Introduction}\label{sect:intro}

Spanning trees can be viewed as spanning subgraphs attaining a trivial
lower bound in size while preserving the connectivity of a given
undirected graph. Arborescences play a similar role for directed
graphs: they preserve rooted connectivity while attaining the same
trivial lower bound in size on such a spanning subgraph. In both
cases, finding the minimum weight of a spanning tree or arborescence
is also possible in strongly polynomial time (and the corresponding
algorithms are covered by introductory graph theory courses, see
\cite{Schrijver}).

A similar phenomenon also occurs for some more complex connectivity
related notions. Assume that the digraph $D=(V+r,E)$ is rooted
$k$-edge-connected, that is, $\lambda_D(r,v)\ge k$ holds for every
$v\in V$, where $\lambda_D(r,v)$ denotes the local edge-connectivity
from $r$ to $v$, that is, the maximum number of pairwise edge-disjoint
paths from the root $r$ to $v$ in $D$. Then the size of every spanning
subgraph of $D$ that preserves rooted $k$-edge-connectivity is
obviously at least $k|V|$ since the in-degree of every $v\in V$ is at
least $k$ in every such subgraph. It follows from Edmonds' classic
disjoint arborescences theorem \cite[Corollary 53.1b]{Schrijver} (see
Theorem~\ref{thm:Edmonds}) firstly that this trivial lower bound can
be attained and secondly that finding the minimum weight of such a
subgraph can be reduced to weighted matroid intersection and thus it
can also be solved in strongly polynomial time \cite[Section
  41.3]{Schrijver}. A different, polyhedral approach to solving the
same problem was given in \cite{Frank09}.

On the other hand, for the notions of strong connectivity of digraphs
or $k$-connectivity (either edge or vertex) of graphs finding even a
minimum size of a spanning subgraph preserving these connectivity
requirements is NP-hard as they contain the Hamilton-cycle problem.

In this paper we focus on a type of connectivity notion, rooted local
edge-connectivity of digraphs, which resembles the former ones listed
above in the sense that a trivial lower bound on the size of a
subgraph having this property is known to be attainable. On the other
hand, the complexity of finding the minimum weight of such a subgraph
is unknown \cite{Joo21}.

The following theorem was proved by Lov\'asz \cite{Lovasz73}. Note
that in this paper digraphs with a designated root node $r$ will be
denoted as $D=(V+r,E)$ meaning that $V$ does not contain $r$ and thus
the full node set of $D$ is $V+r=V\cup\{r\}$.

\begin{thm}[\cite{Lovasz73}]\label{thm:Lovasz}
  For every directed graph $D=(V+r,E)$ with a root node $r$ there
  exists a spanning subgraph $F$ such that
  $\lambda_D(r,v)=\lambda_F(r,v)=\varrho_F(v)$ holds for every vertex
  $v\in V$.
\end{thm}

Since $\varrho_F(v)\ge\lambda_D(r,v)$ must obviously hold for the
in-degree $\varrho_F(v)$ of the vertex $v$ if the rooted local
edge-connectivity from $r$ to $v$ is preserved by the spanning
subgraph $F$, it follows that $\sum_{v\in V}\lambda_D(r,v)$ is a lower
bound on the size of a spanning subgraph that preserves all rooted
local edge-connectivity values. The above theorem claims that this
trivial lower bound can be attained for every directed graph $D$ and
the proof given in \cite{Lovasz73} also implies a polynomial algorithm
for finding such a subgraph.

In this paper we address the weighted version of this problem: finding
the minimum weight of a spanning subgraph among the ones guaranteed by
Theorem~\ref{thm:Lovasz}. Although the complexity of this problem will
remain to be open for general digraphs \cite{Joo21}, we will solve it
for acyclic digraphs -- an important class of digraphs for which this
problem still seems not obvious. In particular, we will prove that if
$D$ is acyclic then the edge-sets of the subgraphs $F$ fulfilling the
conditions of Theorem~\ref{thm:Lovasz} form the set of bases of a
matroid. This fact will yield a strongly polynomial algorithm for
finding such a subgraph.

Theorem~\ref{thm:Lovasz} was rediscovered independently in
\cite{Calvillo-Vives78} and the notion of a \emph{flame} was also
coined:

\begin{dfn}
  The directed graph $F=(V+r,E)$ with root node $r$ is a \emph{flame}
  if $\lambda_F(r,v)=\varrho_F(v)$ holds for every node $v\in V$.
\end{dfn}

Using this notion Theorem~\ref{thm:Lovasz} can also be claimed in
the following form: for every directed graph $D=(V+r,E)$ the maximum
size (that is, number of edges) of a spanning subgraph $F$ that is a
flame is $\sum_{v\in V}\lambda_D(r,v)$. With respect to this, spanning
subgraphs fulfilling the requirements of Theorem~\ref{thm:Lovasz} will
be referred to as \emph{maximal flames} henceforth.

It is an immediate corollary of Edmonds' disjoint arborescences
theorem (see Theorem~\ref{thm:Edmonds}) that if $\varrho_F(v)=k$ holds
for every vertex $v\in V$ of the flame $F=(V+r,E)$ then $F$ is the
union of $k$ edge-disjoint arborescences of $F$ rooted at $r$.  In
this paper we will also prove a generalization of this claim that
holds for every flame with no restriction on the in-degrees
(Theorem~\ref{thm:flamedecomp}).

After its publication in 1973, various possible generalizations of
Edmonds' disjoint arborescences theorem were proposed. For example,
N. Robertson asked if it was true that in every digraph with root node
$r$ there exist a family of partial arborescences rooted at $r$ such
that every $v\in V$ is entered by $\lambda(r,v)$ of them. However,
this is false, a simple counterexample was shown by Lov\'asz in
\cite{Lovasz73}. In fact, this prompted Lov\'asz to prove
Theorem~\ref{thm:Lovasz}, a weakening of this false conjecture --
which, on the other hand, does not imply Edmonds' disjoint
arborescences theorem. In this paper we will use the above mentioned
result on flames to prove a theorem that is a direct generalization of
both Theorem~\ref{thm:Lovasz} and Edmonds' disjoint arborescences
theorem (see Theorem~\ref{thm:main}).

Unless otherwise stated, we follow the notation and terminology of
\cite{Schrijver} in this paper. We also use \cite{Schrijver} as a
source of citations for results considered to be classic and widely
known. Note that spanning subgraphs of a directed graph $D=(V+r,E)$
will be identified with their set of edges; correspondingly, subflames
$F$ of $D$ will also be thought of as edge sets $F\subseteq E$.

This paper is structured as follows. Preliminary results on local
rooted edge-connectivity and flames will be given in
Section~2. Sections~3 and 4 will be dedicated to the new results
mentioned above. Finally, Section 5 concludes the paper.

\section{Preliminary results}

Assume that $D=(V+r,E)$ is a directed graph with a root node $r$ and
$v\in V$ is arbitrary. The set of vertices $X\subseteq V$ is called
\emph{$v$-tight} if $v\in X$ and $\varrho(X)=\lambda_D(r,v)$ (where
$\varrho(X)=\varrho_D(X)$ denotes the in-degree of $X$ in $D$). In
other words, $X$ is $v$-tight if $\varrho(X)$ is minimum among all
sets $X\subseteq V$ containing $v$ (by Menger's theorem
\cite[Corollary 9.1b]{Schrijver}). It follows from standard
submodularity technique that $v$-tight sets are closed under union and
intersection; indeed, if $\lambda_D(r,v)=k$ and $X$ and $Y$ are
$v$-tight then $k+k=\varrho(X)+\varrho(Y)\ge\varrho(X\cup
Y)+\varrho(X\cap Y)\ge k+k$ and hence $\varrho(X\cup Y)=\varrho(X\cap
Y)=k$. This implies that there exist both a unique largest and a
unique smallest $v$-tight subset of $V$ which will be denoted by
$M^D_v$ and $T^D_v$, respectively (omitting the superscript $D$ if
unambiguous).

The following simple lemma is also a direct consequence of
submodularity (and its claims are already present in \cite{Lovasz73},
some implicitly).

\begin{lem}\label{lem:trivsubmod}
  Let $D=(V+r,E)$ be a digraph with root node $r$ and $u,v\in V$
  arbitrary vertices.
\begin{enumerate}
 \item[(1)] If $X$ is $u$-tight, $Y$ is $v$-tight and $u\in Y$ then $X\cap Y$ is
   $u$-tight and $X\cup Y$ is $v$-tight.
 \item[(2)] If $u\in M_v$ then $M_u\subseteq M_v$.
 \item[(3)] If $u\in T_v$ then $T_u\subseteq T_v$.
\end{enumerate}
\end{lem}
\begin{proof}
  Let $k=\varrho(X)$ and $\ell=\varrho(Y)$. We get
  $k+\ell=\varrho(X)+\varrho(Y)\ge\varrho(X\cup Y)+\varrho(X\cap Y)\ge
  k+\ell$ from submodularity using $v\in X\cup Y$ and $u\in X\cap Y$
  which immediately proves claim \emph{(1)}. Then both claims
  \emph{(2)} and \emph{(3)} are implied by \emph{(1)}. In case
  of \emph{(2)}, apply \emph{(1)} with $X=M_u$ and $Y=M_v$; then
  $M_u\subseteq M_v$ is implied by the $v$-tightness of $X\cup Y$ and the
  maximality of $M_v$. Analogously, to prove \emph{(3)} apply
  \emph{(1)} with $X=T_u$ and $Y=T_v$; then
  $T_u\subseteq T_v$ is implied by the $u$-tightness of $X\cap Y$ and the
  minimality of $T_u$.
\end{proof}

The following lemma is the main tool of \cite{Lovasz73} for proving
Theorem~\ref{thm:Lovasz}.

\begin{lem}\label{lem:T_v}
Let $D=(V+r,E)$ be a digraph with root node $r$ and $e=(u,v)$ an edge
of $D$. Then the following are equivalent:
\begin{enumerate}
 \item[(1)] $T_v$ is not entered by $e$ (that is, $u\in T_v$).
 \item[(2)] No $w$-tight set is entered by $e$ for any $w\in V$
    (that is, the deletion of $e$ does not affect the value of
   $\lambda(r,w)$ for any $w\in V$).
\end{enumerate}
\end{lem}
\begin{proof}
  \emph{(2)} implies \emph{(1)} by the definition of $T_v$. To show
  that \emph{(1)} implies \emph{(2)}, let $u\in T_v$ and assume
  towards a contradiction that $e=(u,v)$ enters a $w$-tight set $Y$
  for some $w\in V$. Then since $v\in Y$, $Y$ is $w$-tight and $T_v$
  is $v$-tight, it follows from Lemma~\ref{lem:trivsubmod} that
  $T_v\cap Y$ is also $v$-tight. This contradicts the minimality of
  $T_v$ by $u\notin T_v\cap Y\subset T_v$.
\end{proof}

It may be worth noting that this lemma was already sufficient for
proving Theorem~\ref{thm:Lovasz} in \cite{Lovasz73}. Indeed, if a
spanning subgraph $F$ of minimum size preserving all local rooted
edge-connectivity values $\lambda_D(r,v)$ contained a vertex $v\in V$
with $\varrho_F(v)>\lambda_D(r,v)$ then an edge $e=(u,v)$ had to exist
with $u\in T^F_v$ by $\varrho_F(T^F_v)=\lambda_D(r,v)$, contradicting
the minimality of $F$ by Lemma~\ref{lem:T_v}. The following lemma was
proved in \cite{Joo21} and it happens to yield another simple proof of
Theorem~\ref{thm:Lovasz}. (Here we give a shorter proof of the lemma
than that of \cite{Joo21} that does not involve the engineering of
paths.)

\begin{lem}[\cite{Joo21}]\label{lem:flame+e}
  Assume that $F=(V+r,E)$ is a flame and $v\in V$ is arbitrary. If a
  new edge is added to $E$ that enters $M^F_v$ then the obtained
  digraph $F'$ is also a flame.
\end{lem}
\begin{proof}
  Assume that the newly added edge is $e=(u,w)$ with $u\notin M^F_v$
  and $w\in M^F_v$. Since the addition of $e$ obviously can not
  decrease the value of $\lambda(r,x)$ for any vertex $x$, nor does it
  affect the in-degree of any other vertex than $w$, it only remains
  to be shown that $\lambda_{F'}(r,w)=\lambda_F(r,w)+1$. This readily
  follows from claim \emph{(2)} of Lemma~\ref{lem:trivsubmod}:
  $M^F_w\subseteq M^F_v$ by $w\in M^F_v$, which implies that the
  addition of $e$ increases the in-degree of every $w$-tight set by
  $u\notin M^F_v$ and the definition of $M^F_w$. Hence
  $\lambda_{F'}(r,w)=\lambda_F(r,w)+1$ by Menger's theorem as claimed.
\end{proof}

As mentioned above and as observed in \cite{Joo21}, this lemma also
implies Theorem~\ref{thm:Lovasz}. Indeed, if a subflame $F$ of $D$ of
maximum size contained a vertex $v\in V$ with
$\lambda_F(r,v)<\lambda_D(r,v)$ then an edge of $D$ not in $F$ would
obviously have to enter $M^F_v$ contradicting the maximality of $F$ by
Lemma~\ref{lem:flame+e}. Besides that, this lemma also implies the
following theorem. Recall that a set system $\mathcal{F}\subseteq 2^S$
with $\emptyset\in\mathcal{F}$ is called a \emph{greedoid} on the
(finite) ground set $S$ if it fulfills the augmentation property of
matroids: $X,Y\in\mathcal{F}$, $|X|<|Y|$ imply the existence of an
$y\in Y\setminus X$ such that $X+y\in\mathcal{F}$. In other words, the
definition of a greedoid is obtained from that of a matroid by
relaxing subclusiveness, meaning that subsets of members of
$\mathcal{F}$ need not necessarily be in $\mathcal{F}$. (Nothing
beyond its definition will be relied on about greedoids in this paper,
but readers interested in the topic are referred to \cite{KLS} or
more recently \cite{Szeszler21}.)

\begin{thm}[\cite{Joo21}]\label{thm:flamegreedoid}
  Let $D=(V+r,E)$ be a digraph with root node $r$. Then the subflames
  $F$ of $D$ (viewed as edge sets $F\subseteq E$) form a greedoid on
  the ground set $E$.
\end{thm}
\begin{proof}
  Since it is obvious that $\emptyset$ is a flame, it remains to be
  shown that for any two subflames $F_1,F_2$ of $D$ with $|F_1|<|F_2|$
  there exists an $e\in F_2\setminus F_1$ such that $F_1+e$ is also a
  flame. $|F_1|<|F_2|$ implies the existence of a $v\in V$ for which
  $\lambda_{F_2}(r,v)=\varrho_{F_2}(v)>\varrho_{F_1}(v)=\lambda_{F_1}(r,v)$. Then
  an edge $e\in F_2\setminus F_1$ must enter the set $X=M^{F_1}_v$ since
  $\varrho_{F_2}(X)\ge\lambda_{F_2}(r,v)>\lambda_{F_1}(r,v)=\varrho_{F_1}(X)$. Then
  $F_1+e$ is indeed a flame by Lemma~\ref{lem:flame+e}.
\end{proof}

We remark that subfalmes of $D$ are not closed under taking subsets
(as shown by trivial examples) so they do not form a matroid in
general.

Finally, we will rely on an equivalent definition of gammoids. Recall
that gammoids form a class of matroids typically defined on a subset
of nodes of a directed graph in terms of vertex-disjoint paths, see
\cite[Section 39.4]{Schrijver} for the details. In this paper,
however, we will use a lesser known, equivalent definition of
gammoids. Assume that a directed graph $D=(V+r,E)$ with a root node
$r$ is given along with a vertex $v\in V$. Then the ground set of the
matroid $\mathcal{G}_D(v)=(S,\mathcal{F}_v)$ is
$S=\operatorname{in}_D(v)$, the set of edges entering $v$ and a subset
$X\subseteq S$ belongs to $\mathcal{F}_v$ if and only if there exist
$|X|$ edge-disjoint paths from $r$ to $v$ such that the last edges of
these paths form $X$. I failed in tracing down the original source of
the equivalency of the two definitions of gammoids, but its proof is
simple (see \cite{Fortier17}). On the other hand, for the sake of
self-containedness we give a short proof of the fact that
$\mathcal{G}_D(v)$ is a matroid as this also follows easily from the
above.

\begin{lem}\label{lem:gammoid}
  The above defined
  $\mathcal{G}_D(v)=(\operatorname{in}_D(v),\mathcal{F}_v)$ is a
  matroid.
\end{lem}
\begin{proof}
  It is straightforward that $\emptyset\in\mathcal{F}_v$ and that
  $Y\subseteq X\in\mathcal{F}_v$ implies $Y\in\mathcal{F}_v$ so it is
  only the augmentation property of matroids that needs to be
  proved. So let $X,Y\in\mathcal{F}_v$ with $k=|X|<|Y|=\ell$ and
  define $D_X$ and $D_Y$ to be the digraphs obtained from $D$ by
  deleting all edges of $\operatorname{in}_D(v)\setminus X$ and
  $\operatorname{in}_D(v)\setminus Y$, respectively from
  $D$. Obviously, $\lambda_{D_X}(r,v)=k$ and $\lambda_{D_Y}(r,v)=\ell$
  by the definition of $\mathcal{G}_D(v)$, the former of which implies
  $\varrho_{D_X}(Z)=k$ for $Z=M^{D_X}_v$. Then there must exist an
  edge $e\in Y\setminus X$ entering $Z$ by
  $k<\ell=\lambda_{D_Y}(r,v)\le\varrho_{D_Y}(Z)$. Since $e$ enters all
  $v$-tight sets of $D_X$ by the definition of $Z=M^{D_X}_v$, there
  exist $k+1$ edge-disjoint paths from $r$ to $v$ in $D_X+e$ by
  Menger's theorem. Hence $X+e\in\mathcal{F}_v$ which concludes the
  proof.
\end{proof}

\section{Minimum Weight Maximal Flames in Acyclic
  Digraphs}\label{sect:mwmf}

Assume that a directed graph $D=(V+r,E)$ with root node $r$ and a
weight function $w:E\rightarrow\mathbb{Q}_+$ are given. We consider
the problem of finding a maximal flame $F$ in $D$ with minimum total
weight of its edges.

Note that by Theorem~\ref{thm:Lovasz} this problem is equivalent to
that of finding a spanning subgraph of minimum total weight that
preserves all rooted local edge connectivity values $\lambda(r,v)$. We
also remark that this problem is a far-reaching generalization of the
classic problem of finding a minimum weight arborescence rooted at $r$
in $D$ (see \cite[Section 52.1]{Schrijver}); indeed, if a new root $R$
is added to the graph along with a single edge from $R$ to $r$ then
maximal flames of the new graph (with respect to the root $R$)
correspond to arborescences of $D$ rooted at $r$.

While the minimum weight maximal flame problem seems to be a very
natural one, I failed in finding any mention of it until very recently
in \cite{Joo21} where it was raised and mentioned that its complexity
seems to be open. As it was noted in Section~\ref{sect:intro} (and
also in \cite{Joo21}), it is known to be solvable in polynomial time
in the special case where the $\lambda(r,v)$ values are all equal. In
this section we aim at settling the problem for acyclic digraphs.

Let $F$ be a maximal flame in the directed graph $D=(V+r,E)$ and $v\in
V$ an arbitrary vertex. Then, by Theorem~\ref{thm:Lovasz},
$F\cap\operatorname{in}_D(v)$ is a base of the matroid
$\mathcal{G}_D(v)$ defined above before
Lemma~\ref{lem:gammoid}. Consequently, $F$ is a base of
$\mathcal{M}=\cup_{v\in V}\mathcal{G}_D(v)$, the direct sum (or
disjoint union) of the matroids $\mathcal{G}_D(v)$. As one can expect,
the converse is not true, a base of $\mathcal{M}$ is not necessarily a
maximal flame. An example for this is shown in Figure~\ref{fig:fig1}:
it is easy to check that deleting the dashed edges from the directed
graph of the figure yields a base of $\mathcal{M}$ that is not a
flame. On the other hand, we will show below that maximal flames
coincide with bases of $\mathcal{M}$ in acyclic digraphs. This will
follow easily from the following theorem.

\begin{figure}
\centering
\includegraphics{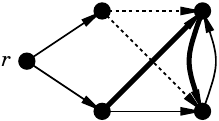}
\caption{A base of $\cup_{v\in V}\mathcal{G}_D(v)$ is not necessarily a flame.}
\label{fig:fig1}
\end{figure}

\begin{thm}\label{thm:sinksaw}
  Assume that in the directed graph $D=(V+r,E)$ with root node $r$ the
  vertex $t\in V$ is a sink (that is, it is not left by any edge).
  Then $F\subseteq E$ is a maximal flame in $D$ if and only if $F$ is
  the disjoint union of a maximal flame in $D-t$ and a base of
  $\mathcal{G}_D(t)$.
\end{thm}
\begin{proof}
  The prove the ``only if'' direction let $F$ be a maximal flame in
  $D$. Since $F\cap\operatorname{in}_D(t)$ is obviously a base of
  $\mathcal{G}_D(t)$, it only remains to be shown that
  $F_0=F\setminus\operatorname{in}_D(t)$ is a maximal flame in
  $D-t$. Obviously, $\lambda_{D-t}(r,v)=\lambda_D(r,v)$ and
  $\lambda_{F_0}(r,v)=\lambda_F(r,v)$ hold for every $v\in V-t$ since
  no path from $r$ to $v$ can pass through the sink $t$. Furthermore,
  $\varrho_{F_0}(v)=\varrho_F(v)$ and
  $\lambda_D(r,v)=\lambda_F(r,v)=\varrho_F(v)$ are also obvious for
  every $v\in V-t$, the latter because $F$ is a maximal flame in
  $D$. All these together show
  $\lambda_{D-t}(r,v)=\lambda_{F_0}(r,v)=\varrho_{F_0}(v)$ for every
  $v\in V-t$ which shows that $F_0=F\setminus\operatorname{in}_D(t)$
  is indeed a maximal flame in $D-t$.

  Turning to the proof of the ``if'' direction, let $F=F_0\cup B$
  where $F_0$ is a maximal flame in $D-t$ and $B$ is a base of
  $\mathcal{G}_D(t)$. Choose $|B|=\lambda_D(r,t)$ edge-disjoint paths
  from $r$ to $t$ in $D$ such that the set of their last edges is $B$
  and let $F_1$ denote the union of the edge-sets of these paths. Then
  $F_1$ is obviously a flame by definition. By
  Theorem~\ref{thm:flamegreedoid} this implies the existence of a
  maximal flame $F_2$ in $D$ such that $F_1\subseteq F_2$; indeed,
  $F_2$ can be obtained by choosing an arbitrary maximal flame in $D$
  and using it to repeatedly apply the augmentation property of
  greedoids starting out from $F_1$. Obviously, since
  $|B|=\lambda_D(r,t)=\varrho_{F_1}(t)$ and therefore no edge of
  $\operatorname{in}_D(t)$ was added during the repeated
  augmentations, $F_2\cap\operatorname{in}_D(t)=B$.

  Since $F_0$ is also a flame in $D$, we can again apply the
  augmentation property of greedoids repeatedly starting out from
  $F_0$ and always choosing the next augmenting edge from $F_2$ until
  a maximal flame $F_3$ is obtained.  We claim that $F_3=F$. Since
  $F_0$ is a maximal flame in $D-t$,
  $\lambda_D(r,v)=\lambda_{F_0}(r,v)=\varrho_{F_0}(v)$ holds for every
  $v\in V-t$ (where we again used
  $\lambda_{D-t}(r,v)=\lambda_D(r,v)$). This implies that no edge of
  $D-t$ was added when $F_3$ was obtained from $F_0$ and hence
  $F_3\setminus\operatorname{in}_D(t)=F_0$. On the other hand, since
  all augmenting edges came from $F_2$ and
  $F_2\cap\operatorname{in}_D(t)=B$, this means that $F_3$ was
  obtained from $F_0$ simply by adding all edges of $B$ one by
  one. This indeed shows $F_3=F$ and thus settles the proof.
\end{proof}

\begin{cor}\label{cor:dagmaxflame}
  Assume that $D=(V+r,E)$ is an acyclic directed graph with root node
  $r$. Then $F\subseteq E$ is a maximal flame in $D$ if and only if
  $F$ is a base of the matroid $\mathcal{M}=\cup_{v\in
    V}\mathcal{G}_D(v)$.
\end{cor}
\begin{proof}
  As mentioned above, the ``only if'' direction is obvious (and it
  holds for all directed graphs) so it is only the ``if'' direction
  that needs to be proved. So let $F$ be a base of
  $\mathcal{M}$. Since $D$ is acyclic, it contains a sink
  $t$. Obviously, $\mathcal{G}_{D-t}(v)=\mathcal{G}_D(v)$ holds for
  every $v\in V-t$ since no path from $r$ to $v$ in $D$ can include
  $t$. Therefore by induction on the number of vertices,
  $F_0=F\setminus\operatorname{in}_D(t)$ is a maximal flame in $D-t$
  (since $D-t$ is acyclic and $F_0$ is a base of the matroid
  $\cup_{v\in V-t}\mathcal{G}_{D-t}(v)$). Hence $F$ is a maximal flame
  in $D$ by Theorem~\ref{thm:sinksaw}.
\end{proof}

The above corollary yields the following algorithmic result.

\begin{thm}
  Let $D=(V+r,E)$ be an acyclic directed graph with root node $r$ and
  $w:E\rightarrow\mathbb{Q}_+$ a weight function. Then a minimum
  weight maximal flame in $D$ can be found in $O\left(|E|^2\right)$
  time.
\end{thm}
\begin{proof}
  By Corollary~\ref{cor:dagmaxflame} a minimum weight base of the
  matroid $\mathcal{M}=\cup_{v\in V}\mathcal{G}_D(v)$ is to be found.
  In other words, we need to find a minimum weight base of the matroid
  $\mathcal{G}_D(v)=(\operatorname{in}_D(v),\mathcal{F}_v)$ for every
  $v\in V$.

  While this problem could be viewed as a minimum cost integer flow
  problem (which would in itself imply strongly polynomial time
  solvability), we get a better running time by implementing the
  greedy algorithm for matroids (see \cite[Section
    40.1]{Schrijver}). This means that we maintain a set
  $X\in\mathcal{F}_v$ starting from $X=\emptyset$. Then, proceeding in
  increasing order of the weights of the edges in
  $\operatorname{in}_D(v)$, we decide for every edge $e$ if
  $X+e\in\mathcal{F}_v$ holds and we replace $X$ by $X+e$ if it does.

  In order for the decisions on $X+e$ to be made, we also maintain
  along with $X$ a binary valued flow $f:E\rightarrow\{0,1\}$ of
  overall value $|X|$ from $r$ to $v$ such that $e\in X$ holds for
  every $e\in\operatorname{in}_D(v)$ if and only if $f(e)=1$. Then we
  use the special case of the maximum flow algorithm of Ford and
  Fulkerson for finding edge-disjoint paths (see \cite[Section
    9.2]{Schrijver}) to decide if $|X|+1$ edge-disjoint paths exist
  from $r$ to $v$ in the graph $D_{X+e}$ obtained from $D$ by deleting
  all edges of $\operatorname{in}_D(v)$ not in $X+e$. In particular,
  $X+e\in\mathcal{F}_v$ holds if and only if there exists a directed
  path from $r$ to $v$ in the graph obtained from $D_{X+e}$ by
  reversing all edges of $D$ with $f(e)=1$. Since this can be decided
  using breadth first search in $O\left(|E|\right)$ time, this imples
  that a base of $\mathcal{G}_D(v)$ can be found in
  $O\left(\varrho_D(v\right)|E|)$ time. (Obviously, this running time
  is also amply sufficient for sorting the edges of
  $\operatorname{in}_D(v)$ by increasing weight.)

  Consequently, a base of $\mathcal{M}=\cup_{v\in V}\mathcal{G}_D(v)$
  (and thus a minimum weight maximal flame in $D$) can be found in
  $O\left(|E|\cdot\sum_{v\in
    V}\varrho_D(v)\right)=O\left(|E|^2\right)$ time.
\end{proof}

\section{Decomposition of flames into edge-disjoint
   branchings}\label{sect:decomp}

Recall that, given a directed graph $D=(V,E)$, a \emph{branching} is a
subset of the edges $B\subseteq E$ such that disregarding the
directions of the edges $B$ is a forest (that is, it contains no
undirected cycle) and every $v\in V$ is entered by at most one edge of
$B$. A vertex $v\in V$ is a \emph{head} of $B$ if it is entered by a
(single) edge of $B$ and $v$ is a \emph{root} of $B$ if it is not
entered by and edge of $B$, but it is left by at least one edge of
$B$. That is, isolated vertices of the graph $(V,B)$ are not
considered as roots of $B$ (at which point we slightly differ in
terminology from \cite{Schrijver}). Therefore every weak component of
the graph $(V,B)$ is either an isolated vertex or it has a single root
in it. The set of roots and set of heads of $B$ are referred to as the
\emph{root-set} and the \emph{head-set} of $B$, respectively. The
branching $B$ is an \emph{arborescence} if the graph $(V,B)$ is weakly
connected (and thus $B$ has a single root). The following classic
theorem is due to Edmonds from 1973 \cite[Corollary 53.1b]{Schrijver}.

\begin{thm}[Edmonds' disjoint arborescences theorem]
  \label{thm:Edmonds}
  The directed graph $D=(V+r,E)$ with root node $r$ contains $k$
  edge-disjoint arborescences rooted at $r$ if and only if
  $\varrho_D(X)\ge k$ for every $X\subseteq V$.
\end{thm}

Since the condition $\varrho_D(X)\ge k$ for every $X\subseteq V$ is
equivalent to $\lambda_D(r,v)\ge k$ for every $v\in V$ by Menger's
theorem, the following generalization of the above theorem (originally
raised by N. Robertson) would feel natural: for every directed graph
$D=(V+r,E)$ with root node $r$ there exist a set of edge-disjoint
branchings such that $r$ is the single root of all of these branchings
and every vertex $v\in V$ is in the head-set of exactly
$\lambda_D(r,v)$ of these branchings. However, this claim is false, a
simple counterexample was given in \cite{Lovasz73}. This observation
prompted Lov\'asz to prove the weaker claim of
Theorem~\ref{thm:Lovasz} -- which, however, does not imply
Theorem~\ref{thm:Edmonds}. In this section we prove another weakening
of this false conjecture that is strong enough to be a common
generalization of Theorems~\ref{thm:Edmonds} and \ref{thm:Lovasz}.

We start by the following theorem, the proof of which will be a
refinement of Lov\'asz's proof given for Theorem~\ref{thm:Edmonds} in
\cite{Lovasz76}.

\begin{thm}\label{thm:goodbranching}
  Assume that $D=(V+r,E)$ is a directed graph with root node $r$. Let
  $m=\max_{v\in V}\lambda_D(r,v)\ge1$ and $V_m=\{v\in
  V:\lambda_D(r,v)=m\}$. Then there exists a branching $B$ in $D$ such
  that its head-set is $V_m$, $\lambda_{D-B}(r,v)=\lambda_D(r,v)$ for
  every $v\in V\setminus V_m$ and $\lambda_{D-B}(r,v)\ge m-1$ for
  every $v\in V_m$.
\end{thm}
\begin{proof}
  Choose a branching $B$ of maximum size such that its head-set is a
  subset of $V_m$, its root-set is a subset of $(V+r)\setminus V_m$,
  $\lambda_{D-B}(r,v)=\lambda_D(r,v)$ for every $v\in V\setminus V_m$
  and $\lambda_{D-B}(r,v)\ge m-1$ for every $v\in V_m$. The empty set
  fulfills these requirements so $B$ exists. We claim that $|B|=|V_m|$
  and hence the head-set of $B$ is $V_m$ which will complete the
  proof. So assume towards a contradiction that $|B|<|V_m|$ and let
  $Z\subseteq V_m$ consist of those vertices $v\in V_m$ that are not
  entered by an edge of $B$.

  Assume first that $\lambda_{D-B}(r,v)=m$ holds for every $v\in
  Z$. Then let $e=(u,v)$ be an arbitrary edge of $D-B$ that enters
  $Z$; obviously, $e$ exists since $\varrho_{D-B}(Z)\ge m$ must
  hold. We claim that $B+e$ fulfills all the requirements on
  $B$. Indeed, $\lambda_{D-B}(r,v)=m$ for every $v\in Z$ implies
  $\varrho_{D-B}(X)\ge m$ for all $X\subseteq V$, $X\cap
  Z\ne\emptyset$ and hence for every vertex $w$ with
  $\lambda_{D-B}(r,w)\le m-1$, $e$ cannot enter any tight set of $w$
  in $D-B$ implying $\lambda_{D-B}(r,w)=\lambda_{D-B-e}(r,w)$ for
  every such $w$. Furthermore, $\lambda_{D-B-e}(r,w)\ge m-1$ is clear
  for every vertex $w$ with $\lambda_{D-B}(r,w)=m$. Since $B+e$ is
  obviously a branching, its head-set is still a subset of $V_m$ by
  $v\in Z$ and its root-set is either the same as that of $B$, or if
  $u$ is a new root then $u\notin V_m$, we have a contradiction.

  Consequently, we can assume that there exists a $v\in Z$ such that
  $\lambda_{D-B}(r,v)=m-1$. Out of all such vertices choose $v$ such
  that $|T^{D-B}_v|$ is minimum and denote $A=T^{D-B}_v$. Then
  $\varrho_{D-B}(A)=m-1$ by the definition of $T^{D-B}_v$. For every
  vertex $w\in A\cap Z$ we have $T^{D-B}_w\subseteq A$ by
  Lemma~\ref{lem:trivsubmod} which implies $T^{D-B}_w=A$ by the choice
  of $v$. Furthermore, $\varrho_{D-B}(A\cap Z)\ge m$ since
  $\varrho_{D-B}(A\cap Z)=\varrho_D(A\cap Z)$ as no edge of $B$ enters
  $Z$ and $\varrho_D(A\cap Z)\ge m$ by $v\in A\cap Z\subseteq
  V_m$.

  Now $\varrho_{D-B}(A)=m-1$ and $\varrho_{D-B}(A\cap Z)\ge m$
  together imply the existence of an edge $e=(u,w)$ of $D-B$ such that
  $w\in A\cap Z$ and $u\in A\setminus Z$. We again claim that $B+e$
  fulfills all the requirements on $B$ made at the beginning of the
  proof. Indeed, $u,w\in A=T^{D-B}_w$ imply by Lemma~\ref{lem:T_v}
  that $\lambda_{D-B-e}(r,v)=\lambda_{D-B}(r,v)$ holds for every $v\in
  V$. Furthermore, $B+e$ is again obviously a branching and the
  requirements on its head-set and root-set hold the same way as
  above. Therefore we again have a contradiction which completes the
  proof.
\end{proof}

This theorem implies the following result on flames.

\begin{thm}\label{thm:flamedecomp}
  Assume that $F=(V+r,E)$ is a flame with root node $r$ and
  $\varrho_F(r)=0$. Let $m=\max_{v\in V}\varrho_F(v)\ge1$ and
  $V_i=\{v\in V:\varrho_F(v)\ge i\}$ for $i=1,2,\ldots,m$. Then $E$
  can be partitioned into the disjoint union of the branchings
  $B_1,B_2,\ldots,B_m$ such that the head-set of $B_i$ is $V_i$ and
  $B_1\cup B_2\cup\ldots\cup B_i$ is a flame for every
  $i=1,2,\ldots,m$.
\end{thm}
\begin{proof}
  We proceed by induction on $m$. If $m=1$ then after deleting all
  isolated vertices from $F$ we get an arborescence rooted at $r$ so
  the claim is obvious. (Observe that $\varrho_F(v)=0$ for a vertex
  $v\in V$ implies that $v$ is isolated since
  $\varrho_F(w)=\lambda_F(r,w)$ could not hold for the endpoint of an
  edge $e=(v,w)$.) So assume $m\ge 2$ and choose the branching $B_m$
  in $F$ using Theorem~\ref{thm:goodbranching}. Then for every $v\in
  V\setminus V_m$ we have
  $\lambda_{F-B_m}(r,v)=\lambda_F(r,v)=\varrho_F(v)=\varrho_{F-B_m}(v)$
  where the last equation holds since the head-set of $B_m$ is
  $V_m$. Furthermore, for every $v\in V_m$ we have
  $\varrho_{F-B_m}(v)=m-1$ which, together with
  $\lambda_{F-B_m}(r,v)\ge m-1$, also implies
  $\lambda_{F-B_m}(r,v)=m-1$. Therefore $F-B_m$ is a flame and
  $\max_{v\in V}\varrho_{F-B_m}(v)=m-1$ so using induction on $F-B_m$
  completes the proof.
\end{proof}

It is important to point out that Theorem~\ref{thm:flamedecomp} is not
implied by the well-known, stronger version of Edmonds' disjoint
arborescences theorem that gives a necessary and sufficient condition
on the existence of edge-disjoint brancings in a directed graph with
given, prescribed head-sets (see \cite[Theorem
  53.1]{Schrijver}). Obviously, if applied on a flame in the context
of Theorem~\ref{thm:flamedecomp}, it would imply the existence of the
edge-disjoint branchings $B_1,B_2,\ldots,B_m$ such that the head-set
of $B_i$ is $V_i$, but the extra requirement that $B_1\cup
B_2\cup\ldots\cup B_i$ be a flame for every $i$ would not necessarily
hold. For example, the directed graph in Figure~\ref{fig:fig2} is a
flame (with root $r$) and the branchings $B_1,B_2,B_3$ made up of the
dashed, regular and thick edges, respectively, fulfill all other
requirements of Theorem~\ref{thm:flamedecomp}, but $B_1\cup B_2$ is
not a flame.

\begin{figure}
\centering
\includegraphics{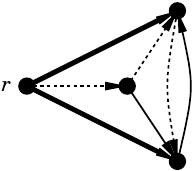}
\caption{Theorem~\ref{thm:flamedecomp} is not implied by Edmonds'
  disjoint branchings theorem.}
\label{fig:fig2}
\end{figure}

Finally, the following theorem is meant to be the main result of
this section.

\begin{thm}\label{thm:main}
  Assume that $D=(V+r,E)$ is a directed graph with root node $r$ and
  let $m=\max_{v\in V}\lambda_D(r,v)\ge1$. Then there exist the
  edge-disjoint branchings $B_1,B_2,\ldots,B_m$ in $D$ such that
  $\lambda_{D_i}(r,v)=\varrho_{D_i}(v)=\min\{\lambda_D(r,v),i\}$ holds
  for every $v\in V$ and $i=1,2,\ldots,m$, where $D_i$ denotes the
  digraph $D_i=(V+r,B_1\cup B_2\cup\ldots\cup B_i)$.
\end{thm}
\begin{proof}
  This follows directly from Theorems~\ref{thm:Lovasz} and
  \ref{thm:flamedecomp} by applying the latter on a maximal flame $F$
  of $D$ given by the former.
\end{proof}

Observe that Theorem~\ref{thm:main} is a common generalization of
Theorems~\ref{thm:Lovasz} and \ref{thm:Edmonds}. Indeed, $B_1\cup
B_2\cup\ldots\cup B_m$ obviously fulfills all requirements of
Theorem~\ref{thm:Lovasz}. Furthermore, if $\lambda_D(r,v)\ge k$ for
every $v\in V$ for some $k$ then $\varrho_{D_k}(v)=k$ holds for every
$v\in V$ which implies that $B_1,\ldots,B_k$ are arborescences rooted
are $r$, so Theorem~\ref{thm:Edmonds} follows.

\section{Conclusions}

In this paper we focused on two separate topics, both related to
flames. In Section~\ref{sect:mwmf} we addressed the problem of finding
a minimum weight maximal flame and we gave an efficient, strongly
polynomial time algorithm for the case where the directed graph is
acyclic. In Section~\ref{sect:decomp} we proved that every flame can
be decomposed into an chain of smaller flames via edge-disjoint
branchings and we used this to give a common generalization of two
well-known, classic results, Edmonds' disjoint arborescences theorem
(Theorem~\ref{thm:Edmonds}) and Lov\'asz's theorem on the existence of a
flame preserving all local edge-connectivity values
(Theorem~\ref{thm:Lovasz}).

While these two topics may not seem to have much in common besides
both being related to flames, it might still be worth pointing out a
more hidden connection: the relation to matroids. In
Section~\ref{sect:mwmf} we relied on the fact that every maximal flame
of the directed graph $D=(V+r,E)$ with root $r$ is a base of the
matroid $\mathcal{M}=\cup_{v\in V}\mathcal{G}_D(v)$ (see the
definition there). While the converse is not true in general, the
major tool of Section~\ref{sect:mwmf} was the observation that it is
true for acyclic directed graphs. Although matroids were not
specifically mentioned in Section~\ref{sect:decomp}, it is still not
hard to observe that they were there implicitly:
Theorem~\ref{thm:flamedecomp} implies that every maximal flame is also
the base of another matroid. Indeed, using the notations of
Theorem~\ref{thm:main}, let $G_i$ denote the undirected graph obtained
from $D$ by disregarding the directions of the edges and contracting
the vertex set $\{r\}\cup\{v\in V:\lambda_D(r,v)<i\}$ for
$i=1,2,\ldots,m$. Then the edge set of the branching $B_i$ becomes a
spanning tree of $G_i$. Therefore Theorem~\ref{thm:flamedecomp}
implies that every maximal flame in $D$ is a base of the matroid
$\mathcal{N}$ obtained as a sum (or union) of the cycle matroids of
the graphs $G_i$ (viewed on the common ground set $E$).

It follows from the above that every maximal flame of $D$ is a common
base of the matroids $\mathcal{M}$ and $\mathcal{N}$ defined above on
the common ground set $E$. Unfortunately though, the converse is not
true as shown by the example of Figure~\ref{fig:fig1}: deleting the
dashed edges yields a base not only of $\mathcal{M}$ (as also
mentioned in Section~\ref{sect:mwmf}), but also of $\mathcal{N}$ since
the remaining edges can be partitioned into the bases of the cycle
matroids of the graphs $G_1$ and $G_2$ as shown by regular and thick
edges, respectively.

It would be interesting to identify a class of digraphs (wider than
that of acyclic ones) for which common bases of $\mathcal{M}$ and
$\mathcal{N}$ are maximal flames since this would imply that the
minimum weight maximal flame problem is solvable for this class of
graphs in strongly polynomial time via weighted matroid
intersection. (This approach would be a generalization of the one
mentioned in Section~\ref{sect:intro} for the case where
$\lambda_D(r,v)=k$ for every $v\in V$.) Furthermore, it would
obviously be also very interesting to settle the complexity of the
minimum weight maximal flame problem for general directed graphs too.

\end{document}